\documentclass[reqno]{amsart}
\usepackage{amssymb}

\newtheorem{theorem}{Theorem}[section]
\newtheorem{lemma}[theorem]{Lemma}
\theoremstyle{definition}

\theoremstyle{remark}
\newtheorem{remark}[theorem]{Remark}

\newcommand{\dontprint}[1]\relax

\newcommand{\ot}{\otimes}

\renewcommand{\SS}{{\mathcal S}}

\newcommand{\si}{\sigma}

\newcommand{\sub}{\subset}

\newcommand{\ov}{\overline}

\newcommand{\rk}{{\operatorname{rk}}}

\renewcommand{\k}{\mathbf{k}}

\newcommand{\End}{{\operatorname{End}}}

\usepackage{xcolor}
\usepackage{amscd}

\title{Almost invariant subspaces and operators}

\author{David Kazhdan}
\author{Alexander Polishchuk}
\thanks{A.P. is partially supported by the NSF grant DMS-2001224, 
and within the framework of the HSE University Basic Research Program and by the Russian Academic Excellence Project `5-100'.}
\dedicatory{To Victor Guillemin, with admiration}

\address{Einstein Institute of Mathematics,
The Hebrew University of Jerusalem,
Jerusalem 91904, Israel}
\email{kazhdan@math.huji.ac.il}
\address{
    Department of Mathematics, 
    University of Oregon, 
    Eugene, OR 97403, USA; National Research University Higher School of Economics; and Korea Institute for 
    Advanced Study 
  }
  \email{apolish@uoregon.edu}

\begin{document}

\begin{abstract}
We give an elementary proof of an efficient version of the Wagner's theorem on almost invariant subspaces (see \cite{Wagner}) and deduce some consequences in the context of Galois extensions.
\end{abstract}

\maketitle

\section{Introduction}

The goal of this note is to present an elementary proof of the following linear algebra result
and to consider its consequences in the context of Galois extensions (see Theorem B). In \cite{KP} this result is applied to the study of the Schmidt rank of quartic polynomials.


\medskip

\noindent
{\bf Theorem A}. {\it
Let $V$ be a vector space over a field $\k$, $X$ a collection of subspaces of $V$, $G$ a group acting linearly on $V$ and preserving $X$ setwise.
For $x\in X$ we denote by $A_x\sub V$ the corresponding subspace.
Assume that for some $r\ge 1$ the following condition holds: for any $x,y\in X$ one has $\dim A_x/(A_x\cap A_y)\le r$.
Then there exists a $G$-invariant subspace $W\sub V$, which is a finite sum of some finite intersections of subspaces from $X$, 
such that 
$$\dim W/(W\cap A_x)\le r, \ \ \dim A_x/(W\cap A_x)\le r\cdot (r+1)^{r+1}.$$
}

\medskip

Originally, this theorem was proved by Wagner \cite{Wagner}, using model theory and without an explicit bound on $\dim A_x/(W\cap A_x)$.
To get an explicit bound we combine Wagner's proof with the idea of the proof of Neumann's explicit bound in Bergman-Lenstra's theorem on almost normal subgroups (see \cite{BL}). We conjecture that it should be possible to improve the bound $r\cdot (r+1)^{r+1}$
to a polynomial (possibly linear) function of $r$, at least in the case when $G$ is finite.

In section 3 we show that in the special case when $G$ is a finite group interchanging vectors of a basis, and subspaces are spanned by subsets of this bases, one can replace  $r\cdot (r+1)^{r+1}$ by $2r$. This reduces to the known set-theoretic result
(see \cite{Neumann}, \cite{BPP}), for which we give a new short proof assuming that $G$ is finite (see Theorem \ref{set-thm}). 
Our version of this result suggests a formula for the approximating $G$-invariant subspace, for which a better bound (polynomial or even linear in $r$) might hold, see Sec.\ \ref{set-sec} for a discussion.

As a corollary of Theorem A, we obtain the following result about linear subspaces and linear operators that are almost invariant
under the action of the Galois group. This result will be used in our forthcoming study of the Schmidt rank of quartics over non-closed fields.

\medskip

\noindent
{\bf Theorem B}. {\it
Let $E/\k$ be a finite Galois extension with the Galois group $G$, and let $V_0$ be a finite dimensional $\k$-vector space.
Let us set $V=V_0\ot_{\k} E$. We consider the natural action of $G$ on $V$ and the induced action of $G$ on the set of
$E$-linear subspaces of $V$ and on $\End_E(V)$. 

\noindent
(i) Suppose $A\sub V$ is an $E$-linear subspace such that for each $\si\in G$, one has 
$$\dim_E(A/(A\cap \si A))\le r,$$ 
for some $r\ge 0$. Then there exists a $\k$-linear subspace $W_0\sub V_0$ such that for
$W=W_0\ot E\sub V$ one has
$$\dim_E(W/(W\cap A))\le r, \ \ \dim_E(A/(W\cap A))\le r\cdot (r+1)^{r+1}.$$

\noindent
(ii) Suppose $V'_0$ is another finite dimensional $\k$-vector space, $V'=V'_0\ot_{\k} E$, and
$T:V\to V'$ is an $E$-linear operator such that for any $\si\in G$, one has 
$$\rk_E(\si(T)-T)\le r,$$ 
for some $r\ge 0$.
Then there exists a $\k$-linear operator $T_0:V_0\to V'_0$, such that 
$$\rk_E(T-(T_0)_E)\le 2r+r\cdot (r+1)^{r+1},$$
where $(T_0)_E:V\to V'$ is obtained from $T_0$ by the extension of scalars.
}

\medskip

\noindent
{\it Acknowledgment}. We are grateful to  Ehud Hrushovski for pointing out the references and for useful discussions.

\section{Linear algebra results}\label{lin-alg-sec}

\subsection{Proof of Theorem A}
Notation: for a subset $S\sub X$ we set $A_S:=\cap_{x\in S} A_x$.
 
For each $m$, $0\le m\le r$, let $\SS_m$ denote the set of all nonempty subsets $S\sub X$ such that 
$$\dim (A_S+A_x)/A_x \le m \ \text{ for any } x\in X.$$
Note that $S=X$ is contained in $\SS_0$ and $\SS_{m-1}\sub \SS_m$.
Set
$$h(m)=\min_{S\in \SS_m} |S|$$
which is either a natural number or $\infty$. Note that by assumption $h(r)=1$. We also have $h(m-1)\ge h(m)$.

We consider separately two cases.

\medskip

\noindent
{\bf Case 1}. There exists $m$, $1\le m\le r$, such that $h(m)$ is finite and $h(m-1)>(r+1)h(m)+1$. Let us take the maximal $m$ with this property
and set
$$W:=\sum_{S\in \SS_m: |S|=h(m)} A_S.$$
Since for all $m'>m$ we have $h(m'-1)\le (r+1)h(m')+1$, we get
$$h(m)\le (r+1)^{r-m}+\ldots +(r+1)^2+(r+1)+1\le (r+1)^r.$$

Note that since $A_S\sub W$ for some $S\in \SS_m$ with $|S|=h(m)$, we have for any $x\in X$,
$$\dim A_x/A_x\cap W\le \dim A_x/A_x\cap A_S\le r\cdot h(m)\le r\cdot (r+1)^r.$$

Next, we claim that for any $x\in X$ one has $\dim W/W\cap A_x \le r$. Indeed, suppose there exists $x\in X$ such that 
$\dim (W+A_x)/A_x\ge r+1$. Then there exist $r+1$ subsets $S_1,\ldots,S_{r+1}\in \SS_m$ with $|S_i|=h(m)$, such that
$$\dim (\sum_{i=1}^{r+1} A_{S_i} + A_x)/A_x \ge r+1.$$
Let us consider the subset 
$$T:=\{x\}\cup S_1\cup\ldots\cup S_{r+1}\sub X.$$
Then we have $|T|\le (r+1)h(m)+1<h(m-1)$. Hence, $T\not\in \SS_{m-1}$, so there exists an element $y\in X$ such that
$$\dim (A_T+A_y)/A_y \ge m.$$
But for any $i=1\ldots,r+1$, we have 
$$(A_T+A_y)/A_y\sub (A_{S_i}+A_y)/A_y,$$
and the latter space has dimension $\le m$ by the definition of $\SS_m$. Hence, we have
$A_T+A_y=A_{S_i}+A_y$ for any $i$. Therefore, 
$$\sum_{i=1}^{r+1} A_{S_i}\sub A_T+A_y\sub A_x+A_y,$$
so 
$$\dim (\sum_{i=1}^{r+1} A_{S_i} + A_x)/A_x\le \dim (A_x+A_y)/A_x\le r,$$
which is a condtradiction.

Finally, we note that since $W/(W\cap A_S)$ is finite dimensional, for any $S\in \SS_m$, $W$ can be written as a finite sum of some
intersections $A_S$.

\medskip

\noindent
{\bf Case 2}. For each $m=1,\ldots,r$ one has $h(m-1)\le (r+1)h(m)+1$. This implies that
$$h(0)\le (r+1)^r+\ldots +(r+1)^2+(r+1)+1\le (r+1)^{r+1}.$$ 
Note that $\SS_0$ consists of $S$ such that $A_S=\cap_{x\in X} A_x$.
In this case we set
$$W:=\cap_{x\in X} A_x.$$
Then since there exists a subset $S\in \SS_0$ with $|S|=h(0)$, we get
$$\dim A_x/A_x\cap W=\dim A_x/A_x\cap A_S\le r\cdot h(0)\le r(r+1)^{r+1}.$$
This ends the proof.

\subsection{Almost invariant operators}

In this section we use Theorem A to approximate almost $G$-invariant operators by $G$-invariant ones by considering their graphs.

\begin{lemma}\label{operators-cor}
Let $V$ and $V'$ be linear representation of a group $G$, such that any $G$-invariant subspace $W\sub V$ (resp., $W'\sub V'$) of finite codimension (resp., dimension) 
admits a $G$-invariant complement.
Assume that $T:V\to V'$ is a linear operator such that for any $g\in G$, one has $\rk(gTg^{-1}-g)\le r$, for some $r\ge 0$.
Then there exists a $G$-invariant operator $T_0:V\to V'$ such that 
$$\rk(T-T_0)\le 2r+r\cdot (r+1)^{r+1}.$$
\end{lemma}

\begin{proof}
Let $A=A_T\sub V\oplus V'$ denote the graph of $T$, i.e., $A=\{(v,Tv)\}$. Note that if $A$ and $A'$ are graphs of $T$ and $T'$ then
$T\cap T'=\{(v,Tv) \ | \ v\in \ker(T-T')\}$, so 
$$\dim T/T\cap T'=\rk(T-T').$$ 

Thus, the assumption on $T$ implies that $\dim (A/(A\cap gA))\le r$ for all $g\in G$. Applying Theorem A,
we obtain a $G$-invariant subspace $A_0\sub V\oplus V'$ such that 
$$\dim A_0/(A_0\cap A)\le r, \ \ \dim A/(A_0\cap A)\le r(r+1)^{r+1}.$$
Let $p_1:V\oplus V'\to V$ and $p_2:V\oplus V'\to V'$ be the projections.
Set $K:=\{v'\in V' \ |\ (0,v')\in A_0\}\sub V'$ and $I:=p_1(A_0)\sub V$.
Note that the subspaces $K$ and $I$ are $G$-invariant, and $p_2:A_0\to V'$ induces a well defined $G$-invariant linear map 
$$\ov{T}_0:I\to V'/K,$$
such that $A_0$ is the pull-back of the graph $\ov{A}_0\sub I\oplus V'/K$ of $\ov{T}_0$ under the projection $\pi:I\oplus V'\to I\oplus V'/K$.
Furthermore, we have 
$$\dim K\le r, \ \ \dim V/I\le r(r+1)^{r+1}.$$
Hence, we can find a $G$-linear projector $p_I:V\to I$ and a $G$-linear projector $p_C:V'\to C\sub V'$, where $C$ is a $G$-invariant
complement to $K$, such that $p_C(K)=0$. Now we set
$$T_0=p_C\circ \ov{T}_0\circ p_I.$$

Note that $\ov{T}_0$ is obtained as the composition of $T_0|_I$ with the projection $V'\to V'/K$. 
Let $\ov{T}:I\to V'/K$ denote the composition of $T|_I$ with the projection $V'\to V'/K$, and let $\ov{A}\sub I\oplus V'/K$ denote the graph of
$\ov{T}$. It is easy to see that $\pi$ induces a surjection 
$$A_0/(A_0\cap A)\to \ov{A}_0/(\ov{A}_0\cap \ov{A}),$$
so we get
$$\rk(\ov{T}-\ov{T}_0)=\dim \ov{A}_0/(\ov{A}_0\cap \ov{A})\le r.$$
But $\ov{T}-\ov{T}_0$ is obtained from $T-T_0$ by restricting to the subspace $I\sub V$ and composing with the projection $V'\to V'/K$,
so we get
$$\rk(T-T_0)\le \dim K+\dim V/I+\rk(\ov{T}-\ov{T}_0)\le 2r+r(r+1)^{r+1}.$$
\end{proof}


\subsection{Proof of Theorem B}

(i) Let us apply Theorem A to the collection $(\si A)$ of all Galois conjugates of $A$.
Note that the action of $G$ is only $\k$-linear, so we should view this is a collection of $\k$-linear subspaces, and use
the dimension function $\dim_{\k}=[E:\k]\cdot \dim_E$.
However, the resulting $G$-invariant $\k$-subspace $W$ is in the lattice generated by $(\si A)$, so it is actually
an $E$-linear subspace.  But $G$-invariant $E$-linear subspaces of $V$ are precisely subspaces obtained
by extension of scalars from $\k$-linear subspaces of $V_0$. This gives the statement.

\noindent
(ii) We apply the same strategy as in the proof of Lemma \ref{operators-cor}.
First, we find a $G$-invariant $E$-subspace $W$ approximating the graph $A$ of $T$. Then we use the fact that $W$ is an extension
of scalars from $W_0\sub V_0\oplus V'_0$ and construct $\k$-linear projectors $p_C:V'_0\to C\sub V'_0$ and $p_I:V_0\to I$,
as in the proof of Lemma \ref{operators-cor}, where $I=p_1(W_0)$ and $C$ is a complement in $V'_0$ to $(0\oplus V'_0)\cap W_0$.
This gives the required operator defined over $\k$.

\section{Almost invariant subsets}\label{set-sec}

The following result is a more concrete version of a theorem of Neumann \cite{Neumann} (see also \cite{BPP} for related results),
stating that if a subset $A$ in a $G$-set $X$ satisfies $|A\setminus gA|\le r$ for every $g\in G$ and some $r>0$, then there exists  
a $G$-invariant subset $A_0\sub X$ such that 
$$|A\setminus A_0|+|A_0\setminus A|<2r.$$
We make an extra assumption that $G$ is finite but as a bonus we get an explicit construction of $A_0$.

\begin{theorem}\label{set-thm} 
Let $X$ be a set with action of a finite group $G$. Let $A\sub X$ be a subset such that for any $g\in G$ one has
$|A\setminus gA|\le r$ for some $r>0$. For every subset $S\sub G$ we denote $A_S:=\cap_{g\in S} gA.$
Now consider the $G$-invariant subset
$$A_0:=\cup_{|S|>|G|/2} A_S.$$
Then
$$|A\setminus A_0|+|A_0\setminus A|\le 2r.$$
\end{theorem}

\begin{proof} 
We have
$$A\setminus A_0=\cup_{|S|>|G|/2} A\setminus A_S=\{a\in A \ | \ |\{g\in G \ |\ ga\in A\}|\le |G|/2\}.$$
Now let us consider the set 
$$P:=\{(a\in A, g\in G) \ | \ ga\not\in A\}.$$
Considering the fibers of the projection of $P$ to $G$ and using the assumption, we see that
$$|P|\le r\cdot |G|.$$

Now, let us consider the projection $p_A:P\to A$.
For every $a\in A\setminus A_0$, we have 
$$|p_A^{-1}(a)|=|\{g\in G \ |\ ga\not\in A\}|\ge \frac{|G|}{2}.$$
Hence, setting
$$P_1:=p_A^{-1}(A\setminus A_0),$$
we get
\begin{equation}\label{A-A0-ineq}
\frac{|G|}{2}\cdot |A\setminus A_0|\le |P_1|.
\end{equation}

Next, setting $B=X\setminus A$, we observe that 
$$A_0\setminus A=A_0\cap B=\cup_{|S|>|G|/2} A_S\cap B=\{b\in B \ | \ |\{g\in G \ |\ gb\in A\}|>|G|/2\}.$$
Let us consider the projection
$$p_B:P\to B: (a,g)\mapsto ga$$
and set
$$P_2:=p_B^{-1}(A_0\cap B).$$
Since for every $b\in A_0\cap B$, we have
$$|p_B^{-1}(b)|=|\{g\in G \ |\ g^{-1}b\in A\}|>\frac{|G|}{2},$$
we deduce that
\begin{equation}\label{B-A0-ineq}
\frac{|G|}{2}\cdot |A_0\cap B|\le |P_2|.
\end{equation}

Finally, we claim that $P_1$ and $P_2$ do not intersect. Indeed, suppose
$(a,g_0)\in P_1\cap P_2$. Then
$$|\{g\in G \ |\ ga\not\in A\}|\ge \frac{|G|}{2}$$
and 
$$|\{g\in G \ |\ gg_0a\in A\}|> \frac{|G|}{2},$$
which is impossible since
$$|\{g\in G \ |\ gg_0a\in A\}|=|\{g\in G \ |\ ga\in A\}|=|G|-|\{g\in G \ |\ gg_0a\in A\}|.$$

Thus, combining \eqref{A-A0-ineq} and \eqref{B-A0-ineq}, we get
$$\frac{|G|}{2}\cdot ( |A\setminus A_0|+|A_0\setminus A|)\le |P_1|+|P_2|\le |P|\le r\cdot |G|,$$
which gives
$$|A\setminus A_0|+|A_0\setminus A|\le 2r.$$
\end{proof}

\begin{remark}
One can see from the proof that the only case when we possibly do not get a strict inequality
$$|A\setminus A_0|+|A_0\setminus A|<2r$$
is when $A_0\sub A$ and $P_1=P$, which is equivalent to $A_S=\cap_{g\in G}gA$ whenever $|S|>|G|/2$.
In this case we can replace $A_0$ with
$$A'_0:=\cup_{|S|\ge |G|/2} A_S.$$
Since for any $S\sub G\setminus \{1\}$ with $|S|\ge |G|/2$, we have $A\cap A_S=\cap_{g\in G}gA$,
assuming that $A$ is not $G$-invariant, we get $A\setminus A'_0\neq\emptyset$.
In this case, running the similar argument to the above proof, we get that
$$|A\setminus A'_0|+|A'_0\setminus A|<2r.$$
\end{remark}

\medskip

Theorem \ref{set-thm} suggests that in the linear algebra setup with $A$ a linear subspace of a $G$-representation $V$,
such that $\dim(A/A\cap gA)\le r$, for finite $G$, one can try to define the approximating $G$-invariant subspace as
\begin{equation}\label{G-inv-subspace-eq}
A_0:=\sum_{|S|>|G|/2} A_S,
\end{equation}
where $A_S=\cap_{g\in S} gA$.

\medskip

\noindent
{\bf Question.} {\it Does there exist a polynomial (or even a linear) function $c(r)$
 such that for $(V,G,A)$ as above and $A_0$ given by \eqref{G-inv-subspace-eq} we have 
$$\dim A/(A_0\cap A) \leq c(r)\ \text{ and } \ \dim A_0/(A_0\cap A)\leq c(r)?$$
}

\end{document}